\documentclass[draft,onecolumn]{IEEEtran}
\IEEEoverridecommandlockouts
\usepackage{cite}
\usepackage{amsmath,amssymb,amsfonts,amsthm}
\usepackage{algorithmic}
\usepackage{graphicx}
\usepackage{textcomp}
\usepackage{xcolor}
\def\BibTeX{{\rm B\kern-.05em{\sc i\kern-.025em b}\kern-.08em
    T\kern-.1667em\lower.7ex\hbox{E}\kern-.125emX}}
\newtheorem{theorem}{Theorem}
\newtheorem{definition}{Definition}

\newtheorem{lemma}{Lemma}

\usepackage{geometry}
\usepackage{setspace}
\begin{document}

\title{  Uncertainty principles for the windowed offset linear canonical transform }
\author{\IEEEauthorblockN{Wen-Biao, Gao$^{1,2}$,
		Bing-Zhao, Li$^{1,2,\star}$,}\\
	\IEEEauthorblockA{
		1. School of Mathematics and Statistics, Beijing Institute of Technology, Beijing 102488, P.R. China}\\
	2. Beijing Key Laboratory on MCAACI, Beijing Institute of Technology, Beijing 102488, P.R. China\\
$^{\star}$Corresponding author: li$\_$bingzhao@bit.edu.cn
}

\maketitle

\begin{abstract}
\label{abstract}
\begin{spacing}{1.25}
The windowed offset linear canonical transform (WOLCT) can be identified as a generalization of the windowed linear canonical transform (WLCT). In this paper, we generalize several different uncertainty principles for the WOLCT, including Heisenberg uncertainty principle, Hardy's uncertainty principle, Beurling's uncertainty principle, Lieb's uncertainty principle, Donoho-Stark's uncertainty principle, Amrein-Berthier-Benedicks's uncertainty principle, Nazarov's uncertainty principle and Logarithmic uncertainty principle.
\end{spacing}
\end{abstract}


\noindent  {\bf 1. Introduction.}

The linear canonical transform(LCT)\cite{xu2013linear,healy2015linear,kou2012windowed,zhang2018novel,urynbassarova2016wigner,fu2008generalized} has become a valuable and useful tool in several fields, such as applied mathematics, signal
processing, optical system analysis, phase retrieval and pattern recognition \cite{kou2013generalized,kou2012paley,healy2015linear,erden1999repeated,fu2008generalized}. It can be used to solve equations, the parameter estimation, sampling and filtering of signal analysis \cite{healy2015linear,kou2013generalized,erden1999repeated,fu2008generalized,wolf2013integral,kou2012windowed,tao2008sampling,jing2018higher}. The offset linear canonical transform (OLCT) \cite{kou2013generalized,zhi2016generalized,stern2007sampling} with six parameters $(a,b,c,d,u_{0},w_{0})$ is a class of linear integral transform. It is a time-shifted and frequency-modulated generalized version of the linear canonical transform (LCT) with four  parameters $(a,b,c,d)$. The time shifting $u_{0}$ and the frequency modulation $w_{0}$ are the two extra parameters, and they make the OLCT more general and flexible than the LCT. It is a useful tool and plays an important role in radar system analysis and filter design. As we know, many linear transforms such as the Fourier transform (FT) \cite{bracewell1986fourier}, the offset FT \cite{pei2007eigenfunctions}, the Fresnel transform(FRST) \cite{james1996generalized}, the fractional FT (FrFT) \cite{tao2009short}, the offset FrFT \cite{james1996generalized}, the LCT and pulse chirping are special cases of the OLCT. Therefore, it is worthwhile to study relevant theory for OLCT.

As a mathematical tool, the OLCT has a wide range of applications\cite{wolf2013integral,kou2012windowed,tao2008sampling}.  But the OLCT has limitations. Because of its global kernel, the OLCT can't display the local OLCT-frequency contents. The windowed function associated with the LCT (WLCT) have attracted wide attention in many literatures \cite{kou2012windowed,kou2012paley,bahri2016some}. The WLCT is method devised to study signals whose spectral content changes with time. They discussed some important properties of the WLCT. For example, the Poisson summation formula, sampling formulas, covariance property, orthogonality property, Paley-Wiener theorem and uncertainly relations. H. Huo proposes the windowed offset linear canonical transform (WOLCT) \cite{huoh2019uncertainty}, the WOLCT by replacing the LCT kernel with the OLCT kernel.
This extension makes the WOLCT more general and flexible than the WLCT \cite{kou2012windowed,kou2012paley}. Several basic properties \cite{wen2019convolution} of the WOLCT are derived. But the uncertainty principles for the WOLCT have not been studied.
The purpose of this paper are to derive some uncertainty principles for the WOLCT. This also provides a foundation for future practical applications.

The paper is organized as follows: In Section 2, we review some definitions. Some different uncertainty principles associated with the WOLCT are provided in Section 3.
In Section 4, the conclusions are drawn.

\noindent  {\bf 2. Preliminary.}

This section presents some relevant contents.

For $1\leq p\leq\infty$, the Lebesgue space $L^{p}(\mathbb{R})$ is defined as the space of all measurable functions on $\mathbb{R}$ such that

\begin{eqnarray}
\|f\|_{L^{p}(\mathbb{R})}=\left(\int_{\mathbb{R}}|f(t)|^{p}\rm{d}\mit t\right)^{\frac{1}{p}}<\infty
\end{eqnarray}
Now we introduce an inner product of the functions $f,g$ defined on $L^2(\mathbb{R})$ is given by
\begin{eqnarray}
\langle f,g\rangle_{L^2(\mathbb{R})}=\int_{\mathbb{R}}f(t)\overline{g(t)}{\rm d}\mit t.
\end{eqnarray}

\begin{lemma}
If $f,g\in L^2(\mathbb{R})$, then the Cauchy-Schwarz inequality \cite{stern2007sampling} holds
\begin{eqnarray}
	\left| \langle f,g\rangle_{L^2(\mathbb{R})}\right|^2\leq \|f\|_{L^2(\mathbb{R})}^2\|g\|_{L^2(\mathbb{R})} ^2
\end{eqnarray}
If and only if $f=-\lambda g$ for some $\lambda \in\mathbb{R}$, the equality holds.
\end{lemma}

For every choice of $\alpha$ and $\beta$ of non-negative integers, the Schwartz space is defined by
\begin{eqnarray}
		S(\mathbb{R})=\left\{f\in C^{\infty}(\mathbb{R}):\sup_{t\in \mathbb{R}}|x^{\alpha}D^{\beta}f(t)|<\infty \right\}
\end{eqnarray}
where $C^{\infty}(\mathbb{R})$ is the set of smooth functions on $\mathbb{R}$ and $D^{\beta}=(\frac{\partial}{\partial t})^{\beta}$.

\begin{definition} (FT) \cite{hardy1933theorem,bracewell1986fourier}
For $f\in L^{2}(\mathbb{R})$, the Fourier transform is given by
\begin{eqnarray}
		F\{f(t)\}(u)=\int_{\mathbb{R}}e^{-i2\pi tu}f(t)\rm{d}\mit t
\end{eqnarray}
\end{definition}

\begin{definition}(OLCT)\cite{xu2015reconstruction,xu2015spectral,kou2013generalized}
Let $A=(a,b,c,d,u_{0},w_{0})$ be a matrix parameter satisfying $a,b,c,d,u_{0},w_{0}\in \mathbb{R}$, and $ad-bc=1$. The OLCT of a signal $f(t)\in L^{2}(\mathbb{R})$ is defined by
\begin{eqnarray}
			O_{A}f(u)=O_{A}[f(t)](u)=\begin{cases}
		\int_{-\infty}^{+\infty}f(t)K_{A}(t,u)\rm{d}\mit t,   &b\neq0  \\
		\sqrt{d}e^{i\frac{cd}{2}(u-u_{0})^{2}+iuw_{0}}f(d(u-u_{0})),    &b=0
		\end{cases}
\end{eqnarray}
where
\begin{equation}
	K_{A}(t,u)=\frac{1}{\sqrt{i2\pi b}}e^{i\frac{a}{2b}t^{2}-i\frac{1}{b}t(u-u_{0})-i\frac{1}{b}u(du_{0}-bw_{0})+i\frac{d}{2b}(u^{2}+u_{0}^{2})}
\end{equation}
\end{definition}
From Def. 2 it can be seen that for case $b = 0$ the OLCT is simply a time scaled version off multiplied by a linear chirp.
Hence, without loss of generality, we assume $b\neq 0$.

The inverse of an OLCT with parameters $A=(a,b,c,d,u_{0},w_{0})$ is given by an OLCT with parameters $A^{-1}=(d,-b,-c,a,bw_{0}-du_{0},cu_{0}-aw_{0})$.
The exact inverse OLCT \cite{pei2007eigenfunctions} expression is
\begin{eqnarray}
	f(t)=O_{A^{-1}}(O_{A}f(u))(t)=e^{i\frac{cd}{2}u_{0}^{2}-iadu_{0}w_{0}+i\frac{ab}{2}w_{0}}\int_{-\infty}^{+\infty}O_{A}f(u)K_{A^{-1}}(u,t)\rm{d}\mit u,
\end{eqnarray}
Next, we introduce one of important properties for the OLCT, its generalized
Parseval formula \cite{bhandari2017shift}, as follows:
\begin{eqnarray}
	\int _{\mathbb{R}}f(t)\overline{g(t)}\rm{d}\mit t=\int _{\mathbb{R}}O_{A}\emph{f}(u))\overline{O_{A}g(u))}\rm{d}\mit u,
\end{eqnarray}
\begin{definition}(WLCT)\cite{bhandari2017shift} Let $\phi\in L^2(\mathbb{R})\backslash \{0\}$ be a window function. The WLCT of a signal $f\in L^2(\mathbb{R})$ with respect to $\phi$ is defined by
\begin{eqnarray}
	G^{A}_{\phi}f(u,w)=\int_{\mathbb{R}}f(t) \overline{\phi(t-w)} \frac{1}{\sqrt{i2\pi b}}e^{i\frac{a}{2b}t^{2}-i\frac{1}{b}tu+i\frac{d}{2b}u^{2}}\rm{d}\mit t
	\end{eqnarray}	
\end{definition}

Next, we give definition of the windowed offset linear canonical transform (WOLCT) and the relationships of the WOLCT with other transforms.

\begin{definition}(WOLCT)\cite{huoh2019uncertainty} Let $\phi\in L^2(\mathbb{R})\backslash \{0\}$ be a window function. The WOLCT of a signal $f\in L^2(\mathbb{R})$ with respect to $\phi$ is defined by
\begin{eqnarray}
	V^{A}_{\phi}f(u,w)=\int_{\mathbb{R}}f(t) \overline{\phi(t-w)} K_{A}(t,u)\rm{d}\mit t
	\end{eqnarray}	
where $K_{A}(t,u)$ is given by (7).
\end{definition}
For a fixed $w$, we have
\begin{eqnarray}
		V^{A}_{\phi}f(u,w)=O_{A}[f(t)\overline{\phi(t-w)}](u)
\end{eqnarray}
When $A=(a,b,c,d,0,0)$, the WOLCT becomes the WLCT.

Using the inverse OLCT to (8), we have
\begin{eqnarray}
		f(t)\overline{\phi(t-w)}=O_{A^{-1}}(V^{A}_{\phi}f(u,w))(t)=
e^{i\frac{cd}{2}u_{0}^{2}-iadu_{0}w_{0}+i\frac{ab}{2}w_{0}}\int_{-\infty}^{+\infty}V^{A}_{\phi}f(u,w)K_{A^{-1}}(u,t)\rm{d}\mit u,
\end{eqnarray}
According to Def.3 and Def.4, we can obtain
\begin{eqnarray}
		V^{A}_{\phi}f(u,w)=e^{i\frac{d}{2b}u_{0}^{2}-i\frac{u}{b}(du_{0}-bw_{0})}G^{A}_{\phi}\left\{f(t)e^{i\frac{1}{b}tu_{0}}\right\}(u,w)
\end{eqnarray}
Using Def.1 and Def.4, the WOLCT can be reduced to the FT
\begin{eqnarray}
		V^{A}_{\phi}f(u,w)=\frac{1}{\sqrt{i2\pi b}}e^{-i\frac{1}{b}u(du_{0}-bw_{0})+i\frac{d}{2b}(u^{2}+u_{0}^{2})}F\{\rho(t)\}(\frac{u}{2\pi b})
\end{eqnarray}
where $\rho(t)=f(t)\overline{\phi(t-w)}e^{i\frac{a}{2b}t^{2}+i\frac{t}{b}u_{0}}$.

\noindent  {\bf 3. Uncertainty Principle for the WOLCT}

The uncertainty principle was first put forward by German physicists Heisenberg \cite{Heisenberg1927} in 1927, also known as Heisenberg Uncertainty Principle. With the deepening of research, the uncertainty principle has been further extended.
There are many different type of uncertainty principles connected with the FT \cite{hardy1933theorem,hogan2007time,bonami2003hermite,hormander1991uniqueness}, for instance Heisenberg's uncertainty principle, Hardy's uncertainty principle and Beurling's uncertainty principle.
Recently, the uncertainty principles associated with the OLCT were proposed \cite{stern2007sampling,huo2019uncertainty}. K. I. Kou \cite{kou2012paley}and M. Bahri \cite{bahri2016some} discussed the uncertainty principles for the WLCT. In view of the WOLCT is a broad version of the WLCT, it is significant and valuable to research uncertainty principles in the WOLCT domain.
So, we study several different kinds of uncertainty principles about the WOLCT in this section.

\noindent  {\bf 3.1 Heisenberg Uncertainty Principle}

\begin{lemma} \cite{bahri2016some,kou2012windowed}
Let $\phi\in L^{2}(\mathbb{R})\backslash \{0\}$. Then, for every $f\in L^{2}(\mathbb{R})$,
we obtain the following consequences:
\begin{eqnarray}
	\int_{\mathbb{R}^{2}}|G_{\phi}^{A}f(u,w)|^{2}\rm{d}\mit w\rm{d}\mit u
=\|\emph{f}\|^{2}\|\mit\phi\|^{2}
\end{eqnarray}
\end{lemma}

\begin{lemma} \cite{wen2019convolution}
 Let $\phi\in L^2(\mathbb{R})\backslash \{0\}$ be window function and $f \in L^2(\mathbb{R})$. Then we get
 \begin{eqnarray}
		\int_{\mathbb{R}}\int_{\mathbb{R}}|V^{A}_{\phi}f(u,w)|^{2}\rm{d}\mit u\rm{d}\mit w=\|\emph{f}\|^{2}\|\mit\phi\|^{2}
\end{eqnarray}
\end{lemma}

Let us give the concept of Heisenberg type uncertainty principle for the WLCT as follows:
\begin{lemma} \cite{bahri2016some}
Let $\phi\in L^{2}(\mathbb{R})\backslash \{0\}$ and $G_{\phi}^{A}f\in L^{2}(\mathbb{R})$. Then, for every $f\in L^{2}(\mathbb{R})$,
we have
\begin{eqnarray}
	\left(\int_{\mathbb{R}^{2}}u^{2}|G_{\phi}^{A}f(u,w)|^{2}\rm{d}\mit w\rm{d}\mit u\right)^{\frac{1}{2}}
\left(\int_{\mathbb{R}}t^{2}|f(t)|^{2}\rm{d}\mit t\right)^{\frac{1}{2}}
\geq\frac{b}{2}\|f\|^{2}\|\phi\|
\end{eqnarray}
\end{lemma}
 Now we derive Heisenberg uncertainty principle for the WOLCT.
\begin{theorem}
Let $\phi\in L^{2}(\mathbb{R})\backslash \{0\}$ and $V_{\phi}^{A}f\in L^{2}(\mathbb{R})$. Then, for every $f\in L^{2}(\mathbb{R})$,
we have
\begin{eqnarray}
	\left(\int_{\mathbb{R}^{2}}u^{2}|V_{\phi}^{A}f(u,w)|^{2}\rm{d}\mit w\rm{d}\mit u\right)^{\frac{1}{2}}
\left(\int_{\mathbb{R}}t^{2}|f|^{2}\rm{d}\mit t\right)^{\frac{1}{2}}
\geq\frac{b}{2}\|f\|^{2}\|\phi\|
		\end{eqnarray}
\end{theorem}
\begin{proof}
Based on Lemma 2, Lemma 4, we get
\begin{eqnarray}
	\left(\int_{\mathbb{R}^{2}}u^{2}|G_{\phi}^{A}f(u,w)|^{2}\rm{d}\mit w\rm{d}\mit u\right)^{\frac{1}{2}}
\left(\int_{\mathbb{R}}t^{2}|(G_{\phi}^{A})^{-1}[G_{\phi}^{A}f](t)|^{2}\rm{d}\mit t\right)^{\frac{1}{2}}
\geq\frac{b}{2}\frac{\int_{\mathbb{R}^{2}}|G_{\phi}^{A}f(u,w)|^{2}\rm{d}\mit w\rm{d}\mit u}{\|\phi\|}
		\end{eqnarray}
Assume that $G_{\phi}^{A}f\in L^{2}(\mathbb{R})$. Since $V_{\phi}^{A}f\in L^{2}(\mathbb{R})$, we
can replace the WLCT of $f$ by the WOLCT of $f$ on the both sides of (20). Then, we have
\begin{eqnarray}
	\left(\int_{\mathbb{R}^{2}}u^{2}|V_{\phi}^{A}f(u,w)|^{2}\rm{d}\mit w\rm{d}\mit u\right)^{\frac{1}{2}}
\left(\int_{\mathbb{R}}t^{2}|(G_{\phi}^{A})^{-1}[V_{\phi}^{A}f](t)|^{2}\rm{d}\mit t\right)^{\frac{1}{2}}
\geq\frac{b}{2}\frac{\int_{\mathbb{R}^{2}}|V_{\phi}^{A}f(u,w)|^{2}\rm{d}\mit w\rm{d}\mit u}{\|\phi\|}
		\end{eqnarray}
From (14), we get $e^{-i\frac{d}{2b}u_{0}^{2}+i\frac{u}{b}(du_{0}-bw_{0})}(G_{\phi}^{A})^{-1}[V^{A}_{\phi}f(u,w)]=f(t)e^{i\frac{1}{b}tu_{0}}$, then
$|f|=|(G_{\phi}^{A})^{-1}[V_{\phi}^{A}f]|$. Using (17), we have
\begin{eqnarray}
	\left(\int_{\mathbb{R}^{2}}u^{2}|V_{\phi}^{A}f(u,w)|^{2}\rm{d}\mit w\rm{d}\mit u\right)^{\frac{1}{2}}
\left(\int_{\mathbb{R}}t^{2}|f|^{2}\rm{d}\mit t\right)^{\frac{1}{2}}
\geq\frac{b}{2}\|f\|^{2}\|\phi\|
		\end{eqnarray}
which completes the proof.  \end{proof}

\noindent  {\bf 3.2 Hardy's Uncertainty Principle}

G.H. Hardy first put forward the Hardy's uncertainty principle in 1933 \cite{hardy1933theorem}. Hardy's uncertainty principle says
that it is impossible for a function and its Fourier transform to decrease very
rapidly simultaneously. The concept of Hardy's uncertainty principle for the FT has been introduced in some literatures such as \cite{hardy1933theorem,hogan2007time}.
\begin{lemma}
Suppose a function $f\in L^{2}(\mathbb{R})$ is such that
\begin{eqnarray}
	|f(t)|\leq C_{1}e^{-\alpha t^{2}}
\end{eqnarray}
and
\begin{eqnarray}
	|F\{f\}(u)|\leq C_{2}e^{-\beta u^{2}}
\end{eqnarray}
for some positive constant $ \alpha, \beta> 0$, $t, u\in \mathbb{R}$ and $C_{1}, C_{2}$  are positive constants, then

(j) If $\alpha\beta>\frac{1}{4}$, then $f=0$.

(jj) If $\alpha\beta=\frac{1}{4}$, then $f(t)=Qe^{-\alpha t^{2}}$, where $Q$ is a constant.

(jjj) If $\alpha\beta<\frac{1}{4}$, then there are infinitely many such functions $f$.
\end{lemma}
By Lemma 3, we derive Hardy's uncertainty principle for the WOLCT.
\begin{theorem}
Let $\phi\in L^2(\mathbb{R})\backslash \{0\}$ be window function, if a function $f\in L^{2}(\mathbb{R})$ is such that
\begin{eqnarray}
	|f(t)|\leq Ce^{-\alpha t^{2}}
\end{eqnarray}
and
\begin{eqnarray}
	|V_{\phi}^{A}f(2\pi ub+u_{0},w)|\leq C_{0}e^{-\beta u^{2}}
\end{eqnarray}
for some positive constant $ \alpha, \beta> 0$, $t, u\in \mathbb{R}$ and $C, C_{0}$  are positive constants, then

(j) If $\alpha\beta>\frac{1}{4}$, then $f=0$.

(jj) If $\alpha\beta=\frac{1}{4}$, then $f(t)=Q \frac{1}{(\overline{\phi(0)})}e^{-\alpha t^{2}}e^{-i\frac{a}{2b}t^{2}-i\frac{t}{b}u_{0}}$, where $Q$ is a constant.

(jjj) If $\alpha\beta<\frac{1}{4}$, then there are infinitely many such functions $f$.
\end{theorem}
\begin{proof}
From (15), we get $\rho(t)=f(t)\overline{\phi(t-w)}e^{i\frac{a}{2b}t^{2}+i\frac{t}{b}u_{0}}$, let $w=t$, then $ \rho(t)=f(t)\overline{\phi(0)}e^{i\frac{a}{2b}t^{2}+i\frac{t}{b}u_{0}}\in L^2(\mathbb{R})$ and $|\overline{\phi(0)}|$ is a positive constant.
Using (25), we get
\begin{eqnarray}
	|\rho(t)|=|f(t)||\overline{\phi(0)}|\leq |\overline{\phi(0)}|Ce^{-\alpha t^{2}}=C_{1}e^{-\alpha t^{2}}
\end{eqnarray}
Applying (15) and (26), we have
\begin{eqnarray}
	|F\{\rho(t)\}(u)|=\sqrt{2\pi b}|V_{\phi}^{A}f(2\pi ub+u_{0},w)|\leq \sqrt{2\pi b}C_{0}e^{-\beta u^{2}}=C_{2}e^{-\beta u^{2}}
\end{eqnarray}
where $C=\frac{C_{1}}{|\overline{\phi(0)}|}$, $C_{0}=\frac{C_{2}}{\sqrt{2\pi b}}$ are some positive constants.

Following from Lemma 3

If $\alpha\beta>\frac{1}{4}$, then $\rho=0$, so f=0. If $\alpha\beta=\frac{1}{4}$, then $\rho(t)=Qe^{-\alpha t^{2}}$, for some constant $Q$.
Hence \\
$f(t)=Q \frac{1}{(\overline{\phi(0)})}e^{-\alpha t^{2}}e^{-i\frac{a}{2b}t^{2}-i\frac{t}{b}u_{0}}$. If $\alpha\beta<\frac{1}{4}$, then there are infinitely many such functions $f$.

Which completes the proof.  \end{proof}

\noindent  {\bf 3.3 Beurling's Uncertainty Principle}

Beurling's uncertainty principle is a more general version of Hardy's uncertainty principle, which is given by A.Beurling and proved by H$\ddot{o}$rmander \cite{bonami2003hermite} and generalized by Bonami et al \cite{hormander1991uniqueness}. The Beurling's uncertainty principle for the OLCT \cite{huo2019uncertainty} can be stated as follows.
\begin{lemma}
Let $f\in L^{2}(\mathbb{R})$ and $O_{A}f\in L^{2}(\mathbb{R})$. If
\begin{eqnarray}
	\int_{\mathbb{R}^{2}}|f(t)O_{A}f(u)|e^{\frac{|tu}{b}|}\rm{d}\mit t\rm{d}\mit u<\infty
\end{eqnarray}
then $f=0.$
\end{lemma}
 According to Lemma 6, we derive Beurling's uncertainty principle for the WOLCT.
\begin{theorem}
Let $\phi\in L^2(\mathbb{R})\backslash \{0\}$ be window function, $f\in L^{2}(\mathbb{R})$ and $V_{A}f\in L^{2}(\mathbb{R})$. If
\begin{eqnarray}
	\int_{\mathbb{R}^{2}}|f(t)V_{A}f(u,w)|e^{\frac{|tu}{b}|}\rm{d}\mit t\rm{d}\mit u<\infty
\end{eqnarray}
then $f=0.$
\end{theorem}
\begin{proof}
Using (30) and $\phi\in L^2(\mathbb{R})\backslash \{0\}$ , we get
\begin{eqnarray}
	\int_{\mathbb{R}^{2}}|f(t)V_{A}f(u,w)|e^{\frac{|tu}{b}|}\rm{d}\mit t\rm{d}\mit u|\overline{\mit\phi(t-w)}|<\infty
\end{eqnarray}
Let $g(t,w)=f(t)\overline{\phi(t-w)}$, since $f\in L^{2}(\mathbb{R})$ and $V_{A}f\in L^{2}(\mathbb{R})$, then $g(t,w)\in L^{2}(\mathbb{R})$ and $O_{A}g\in L^{2}(\mathbb{R})$.

According to (31) and (12), we have
\begin{eqnarray}
	\int_{\mathbb{R}^{2}}|g(t,w)O_{A}g(u)|e^{\frac{|tu}{b}|}\rm{d}\mit t\rm{d}\mit u=
\int_{\mathbb{R}^{2}}|\emph{f}(t)\overline{\mit\phi(t-w)}V_{A}\emph{f}(u,w)|e^{\frac{|tu}{b}|}\rm{d}\mit t\rm{d}\mit u=
\int_{\mathbb{R}^{2}}|\emph{f}(t)V_{A}\emph{f}(u,w)|e^{\frac{|tu}{b}|}\rm{d}\mit t\rm{d}\mit u|\overline{\mit\phi(t-w)}|
<\infty
\end{eqnarray}
Hence, it follows from Lemma 6 that $g = 0$. Therefore, we have $f = 0$.

Which completes the proof.  \end{proof}

\noindent  {\bf 3.4 Lieb's Uncertainty Principle}

Lieb's uncertainty principle was first proposed by Elliott H. Lieb in 1989 \cite{lieb1990integral}. The concept of Lieb's uncertainty principle for the WLCT was derived by M. Bahri, as follow
\begin{lemma} \cite{bahri2016some}
Let $\phi,f\in L^{2}(\mathbb{R})$ and $2\leq p<\infty$. Then
\begin{eqnarray}
	\int_{\mathbb{R}^{2}}|G_{\phi}^{A}f(u,w)|^{p}\rm{d}\mit t\rm{d}\mit u\leq\frac{2}{p}(E_{A})^{p}(\|\emph{f}\|\|\mit\phi\|)^{p}
\end{eqnarray}
where $E_{A}=(2\pi)^{-\frac{1}{2}}|b|^{\frac{1}{p}-\frac{1}{2}}$.
\end{lemma}
 According to Lemma 7, we derive Lieb's uncertainty principle for the WOLCT.
\begin{theorem}
Let $\phi,f\in L^{2}(\mathbb{R})$ and $2\leq p<\infty$. Then
\begin{eqnarray}
	\int_{\mathbb{R}^{2}}|V^{A}_{\phi}f(u,w)|^{p}\rm{d}\mit t\rm{d}\mit u\leq\frac{2}{p}(E_{A})^{p}(\|\emph{f}\|\|\mit\phi\|)^{p}
\end{eqnarray}
where $E_{A}=(2\pi)^{-\frac{1}{2}}|b|^{\frac{1}{p}-\frac{1}{2}}$.
\end{theorem}
\begin{proof}
Let $h(t)=f(t)e^{i\frac{1}{b}tu_{0}}$, then $\|h\|=\|f\|$. Using (14), we get
\begin{eqnarray}
		|V^{A}_{\phi}f(u,w)|=\left|G^{A}_{\phi}\{h(t)\}(u,w)\right|
\end{eqnarray}
Since $f\in L^{2}(\mathbb{R})$, it implies that $h\in L^{2}(\mathbb{R})$. Replacing $f$ in both sides of (33) with
$h$, we have
\begin{eqnarray}
	\int_{\mathbb{R}^{2}}|G_{\phi}^{A}h(u,w)|^{p}\rm{d}\mit t\rm{d}\mit u\leq\frac{2}{p}(E_{A})^{p}(\|\emph{h}\|\|\mit\phi\|)^{p}
\end{eqnarray}
so
\begin{eqnarray}
	\int_{\mathbb{R}^{2}}|V^{A}_{\phi}f(u,w)|^{p}\rm{d}\mit t\rm{d}\mit u\leq\frac{2}{p}(E_{A})^{p}(\|\emph{f}\|\|\mit\phi\|)^{p}
\end{eqnarray}
Which completes the proof.  \end{proof}
The theorem is quite different from the one presented in another paper \cite{huoh2019uncertainty}.

\noindent  {\bf 3.5 Donoho-Stark's Uncertainty Principle and Amrein-Berthier-Benedicks's Uncertainty Principle}

Donoho-Stark's uncertainty principle was first proposed by Donoho and Stark in 1989 \cite{donoho1989uncertainty}. Donoho-Stark's uncertainty principle and Amrein-Berthier-Benedicks's Uncertainty Principle for FT \cite{grochenig2013foundations} have been
proposed. Let us recall the concept of Donoho-Stark's uncertainty principle for FT \cite{grochenig2013foundations} as follows.
\begin{definition} Let $\epsilon\geq0$, a function $f(t)\in L^2(\mathbb{R})$ is $\epsilon-$concentrated on a measurable set $D\subseteq\mathbb{R}$, if
\begin{eqnarray}
	\left( \int_{\mathbb{R}\backslash D}|f(t)|^{2}\rm{d}\mit t\right)^{\frac{1}{2}}\leq\epsilon\|f\|_{2}
	\end{eqnarray}	
If $0\leq\epsilon\leq\frac{1}{2}$. then the most of energy is concentrated on $D$, and $D$ is indeed the essential support of $f$.
If $\epsilon=0$, then $D$ is the exact support of $f$.\\
Similarly, we say that its FT is $\epsilon-$concentrated on a measurable set $S\subseteq\mathbb{R}$, if
\begin{eqnarray}
	\left( \int_{\mathbb{R}\backslash S}|F\{f(t)\}(u)|^{2}\rm{d}\mit u\right)^{\frac{1}{2}}\leq\epsilon\|Ff\|_{2}
	\end{eqnarray}	
\end{definition}
\begin{lemma} \cite{grochenig2013foundations}
Let a function $f(t)\in L^2(\mathbb{R})$, $f\neq0$, is $\epsilon_{D}-$concentrated on a measurable set $D\subseteq\mathbb{R}$, and
its FT $Ff(u)$is $\epsilon_{S}-$concentrated on a measurable set $S\subseteq\mathbb{R}$. Then,
\begin{eqnarray}
	|D||S|\geq(1-\epsilon_{D}-\epsilon_{S})^{2}
		\end{eqnarray}
where $|D|$ and $|S|$ are the measures of the sets $D$ and $S$.
\end{lemma}
Before we obtain the Donoho-Stark's uncertainty principle for the WOLCT, let's give a definition.
\begin{definition} Let $\phi\in L^{2}(\mathbb{R})\backslash \{0\}$, $\epsilon\geq0$, a function $f(t)\in L^2(\mathbb{R})$, its WOLCT $V_{\phi}^{A}f(u,w)$ is $\epsilon_{S}-$concentrated on a measurable set $S\subseteq\mathbb{R}$, if
\begin{eqnarray}
	\left( \int_{\mathbb{R}}\int_{\mathbb{R}\backslash S}|V_{\phi}^{A}\{f(t)\}(u,w)|^{2}\rm{d}\mit u\rm{d}\mit w\right)^{\frac{1}{2}}\leq\epsilon_{S}\|V_{\phi}^{A}f\|_{2}
	\end{eqnarray}	
where $\|V_{\phi}^{A}f\|_{2}=\left( \int_{\mathbb{R}}\int_{\mathbb{R}}|V_{\phi}^{A}\{f(t)\}(u,w)|^{2}\rm{d}\mit u\rm{d}\mit w\right)^{\frac{1}{2}}$
\end{definition}
Now we generalize Donoho-Stark's uncertainty principle for the WOLCT.
\begin{theorem} Let $\phi\in L^{2}(\mathbb{R})\backslash \{0\}$, a function $f(t)\in L^2(\mathbb{R})$, $f\neq0$, is $\epsilon_{D}-$concentrated on a measurable set $D\subseteq\mathbb{R}$, and
its WOLCT is $\epsilon_{S}-$concentrated on a measurable set $S\subseteq\mathbb{R}$. Then,
\begin{eqnarray}
	|D||S|\geq2\pi b(1-\epsilon_{D}-\epsilon_{S})^{2}
		\end{eqnarray}
\end{theorem}
\begin{proof}
According to Definition 4, we can get the WOLCT as follow
\begin{eqnarray}
	V^{A}_{\phi}f(u,w)=\frac{1}{\sqrt{i2\pi b}}e^{-i\frac{1}{b}u(du_0-bw_0)+i\frac{d}{2b}(u^{2}+u_{0}^{2})}F\{\rho(t)\}(u)
		\end{eqnarray}
where
\begin{eqnarray}
	\rho(t)=f(t)T_{w}\overline{\phi(t)}e^{i\frac{a}{2b}t^{2}+i\frac{t}{b}u_{0}},\ \ F\rho(u)=F\{\rho(t)\}(u)=\int_{\mathbb{R}}\rho(t)e^{-i\frac{1}{b}tu}\rm{d}\mit t
		\end{eqnarray}
Using (43), we can have
\begin{eqnarray}
	|V^{A}_{\phi}f(u,w)|=\frac{1}{\sqrt{2\pi b}}|F\rho(u)|
		\end{eqnarray}
Because of WOLCT is $\epsilon_{S}-$concentrated on a measurable set $S\subseteq\mathbb{R}$, hence
\begin{eqnarray}
	\left( \int_{\mathbb{R}}\int_{\mathbb{R}\backslash S}|V_{\phi}^{A}\{f(t)\}(u,w)|^{2}\rm{d}\mit u\rm{d}\mit w\right)^{\frac{1}{2}}\leq\epsilon_{S}\|V_{\phi}^{A}f\|_{2}
	\end{eqnarray}	
From Definition 1, we know that $F\rho(2\pi bu)$ is the FT of $\rho(t)$. Let $w=t$, then $ \rho(t)=f(t)\overline{\phi(0)}e^{i\frac{a}{2b}t^{2}+i\frac{t}{b}u_{0}}\in L^2(\mathbb{R})$ and $|\overline{\phi(0)}|>0$. We have $|\rho(t)|=|f(t)||\overline{\phi(0)}|$.\\
By (38) and (39), we obtain
\begin{eqnarray}
	\left( \int_{\mathbb{R}\backslash S}|F\rho(u)|^{2}\rm{d}\mit u\right)^{\frac{1}{2}}\leq\epsilon_{S}\|F\rho\|_{2}
	\end{eqnarray}	
Hence, $F\rho(2\pi bu)$ is $\epsilon_{S}-$concentrated on a measurable set $\frac{S}{2\pi b}\subseteq\mathbb{R}$. Moreover,
$f(t)$ is $\epsilon_{D}-$concentrated on a measurable set $D\subseteq\mathbb{R}$, so
\begin{eqnarray}
	\left( \int_{\mathbb{R}\backslash D}|f(t)|^{2}\rm{d}\mit t\right)^{\frac{1}{2}}\leq\epsilon_{D}\|f\|_{2}
	\end{eqnarray}	
Using $|\rho(t)|=|f(t)||\overline{\phi(0)}|$, we get
\begin{eqnarray}
	\left( \int_{\mathbb{R}\backslash D}|\rho(t)|^{2}\rm{d}\mit t\right)^{\frac{1}{2}}\leq\epsilon_{D}\|\rho\|_{2}
	\end{eqnarray}	
Therefore, $\rho(t)$ is $\epsilon_{D}-$concentrated on a measurable set $D\subseteq\mathbb{R}$. According to Lemma 8, we have
\begin{eqnarray}
	|D||\frac{S}{2\pi b}|\geq(1-\epsilon_{D}-\epsilon_{S})^{2}
		\end{eqnarray}
Hence,
\begin{eqnarray}
	|D||S|\geq2\pi b(1-\epsilon_{D}-\epsilon_{S})^{2}
		\end{eqnarray}
Which completes the proof.
  \end{proof}
Let's first give Amrein-Berthier-Benedicks's Uncertainty Principle for FT \cite{grochenig2013foundations}, as follow
\begin{lemma}
Let a function $f(t)\in L^2(\mathbb{R})$, $supp(f)\subseteq D$, and $suup(Ff(u))\subseteq S$, where $D,S$ are measurable sets in $\mathbb{R}$. If $|D||S|<+\infty$, then $f(t)=0$.
\end{lemma}
Amrein-Berthier-Benedicks's Uncertainty Principle for WOLCT is obtained as follow.
\begin{theorem} Let a function $f(t)\in L^2(\mathbb{R})$, $supp(f)\subseteq D$, and $suup(V_{\phi}^{A}f)\subseteq S$, where $D,S$ are measurable sets in $\mathbb{R}$. If $|D||S|<+\infty$, then $f(t)=0$.
\end{theorem}
\begin{proof} By Theorem 5 and Lemma 9, we can get the result. \\
Since $suup(V_{\phi}^{A}f)\subseteq S$, (43) and (44), we can get $suup(F\rho(u))\subseteq S$. So we have $suup(F\rho(2\pi bu))\subseteq \frac{S}{2\pi b}$, Furthermore, we obtain $\rho(t)\in L^2(\mathbb{R})$, $supp(\rho(t))\subseteq D$ according to $f(t)\in L^2(\mathbb{R})$, $supp(f)\subseteq D$ and $|\rho(t)|=|f(t)||\overline{\phi(0)}|$. Since $suup(F\rho(2\pi bu))$ is the FT of $\rho(t)$, we have $\rho(t)=0$ by Lemma 9. Finally, we get $f(t)=0$.\\
Which completes the proof.
 \end{proof}

\noindent  {\bf 3.6 Nazarov's Uncertainty Principle}

Nazarov's uncertainty principle was first proposed by F.L. Nazarov in 1993 \cite{nazarov1993local,jaming2007nazarov}. It says what happens if a nonzero
function and its Fourier transform are only small outside a compact set. The following lemma is the concept of Nazarov's uncertainty principle for the OLCT \cite{huo2019uncertainty}.
\begin{lemma} Let $f\in L^{2}(\mathbb{R})$ and $T,U$ be two subsets of $\mathbb{R}$ with finite measure.
Then, there exists a constant $C > 0$, such that
\begin{eqnarray}
	\int_{\mathbb{R}}|f(t)|^{2}\rm{d}\mit t\leq Ce^{C|T||U|}\left( \int_{\mathbb{R}\setminus T}|\emph{f}(t)|^{2}\rm{d}\mit t +\int_{\mathbb{R}\backslash (Ub)}|O_{A}\emph{f}(u)|^{2}\rm{d}\mit u\right)
\end{eqnarray}
\end{lemma}
Inspired by Lemma 10, we obtain Nazarov's uncertainty principle for the WOLCT.
\begin{theorem}
Let $\phi\in L^2(\mathbb{R})\backslash \{0\}$ be window function, $f\in L^2(\mathbb{R})$ and $T,U$ be two subsets of $\mathbb{R}$ with finite measure.
Then, there exists a constant $C > 0$, such that
\begin{eqnarray}
	\|\phi\|^{2}\int_{\mathbb{R}}|f(t)|^{2}\rm{d}\mit t\leq Ce^{C|T||U|}\left( \|\mit\phi\|^{2}\int_{\mathbb{R}\setminus T}|\emph{f}(t)|^{2}\rm{d}\mit t +\int_{\mathbb{R}}\int_{\mathbb{R}\backslash (Ub)}|V_{A}\emph{f}(u,w)|^{2}\rm{d}\mit u\rm{d}\mit w\right)
		\end{eqnarray}
\end{theorem}
\begin{proof}
According to Lemma 10, (12) and $f(t)\overline{\phi(t-w)}\in L^2(\mathbb{R})$, we have
\begin{eqnarray}
	\int_{\mathbb{R}}|f(t)\overline{\phi(t-w)}|^{2}\rm{d}\mit t\leq Ce^{C|T||U|}\left( \int_{\mathbb{R}\setminus T}|\emph{f}(t)\overline{\mit\phi(t-w)}|^{2}\rm{d}\mit t +\int_{\mathbb{R}\backslash (Ub)}|O_{A}(\emph{f}\overline{\mit\phi})(u)|^{2}\rm{d}\mit u\right)
\end{eqnarray}
Hence
\begin{eqnarray}
	\int_{\mathbb{R}}|f(t)|^{2}|\overline{\phi(t-w)}|^{2}\rm{d}\mit t\leq Ce^{C|T||U|}\left( \int_{\mathbb{R}\setminus T}|\emph{f}(t)|^{2}|\overline{\mit\phi(t-w)}|^{2}\rm{d}\mit t +\int_{\mathbb{R}\backslash (Ub)}|V_{A}\emph{f}(u,w)|^{2}\rm{d}\mit u\right)
\end{eqnarray}
 Integrating with respect to $dw$ we get
\begin{eqnarray}
	\int_{\mathbb{R}}\int_{\mathbb{R}}|f(t)|^{2}|\overline{\phi(t-w)}|^{2}\rm{d}\mit t\rm{d}\mit w\leq Ce^{C|T||U|}\left( \int_{\mathbb{R}}\int_{\mathbb{R}\setminus T}|\emph{f}(t)|^{2}|\overline{\mit\phi(t-w)}|^{2}\rm{d}\mit t\rm{d}\mit w +\int_{\mathbb{R}}\int_{\mathbb{R}\backslash (Ub)}|V_{A}\emph{f}(u,w)|^{2}\rm{d}\mit u\rm{d}\mit w\right)
\end{eqnarray}
So
\begin{eqnarray}
	\|\phi\|^{2}\int_{\mathbb{R}}|f(t)|^{2}\rm{d}\mit t\leq Ce^{C|T||U|}\left( \|\mit\phi\|^{2}\int_{\mathbb{R}\setminus T}|\emph{f}(t)|^{2}\rm{d}\mit t +\int_{\mathbb{R}}\int_{\mathbb{R}\backslash (Ub)}|V_{A}\emph{f}(u,w)|^{2}\rm{d}\mit u\rm{d}\mit w\right)
\end{eqnarray}
which completes the proof.  \end{proof}

\noindent  {\bf 3.7 Logarithmic Uncertainty Principle}

W. Beckner first introduced Logarithmic uncertainty principle in 1995 \cite{beckner1995pitt}.
We obtain the concept of Logarithmic uncertainty principle for the WLCT \cite{bahri2016some} as follows:
\begin{lemma}
Let $\phi\in S(\mathbb{R})$. Then, for every $f\in S(\mathbb{R})$,
we have:
\begin{eqnarray}
	\int_{\mathbb{R}^{2}}\ln |u||G_{\phi}^{A}f(u,w)|^{2}\rm{d}\mit w\rm{d}\mit u+\|\mit\phi\|^{2}
\int_{\mathbb{R}}\ln |t||\emph{f}(t)|^{2}\rm{d}\mit t
\geq(M+\ln |b|)\|\emph{f}\|^{2}\|\mit\phi\|^{2}
		\end{eqnarray}
where $M=\varphi(\frac{1}{2}-\ln \pi)$, $\varphi(t)=\frac{\rm{d}}{\rm{d}t}\ln [\Gamma(t)] $ and $\Gamma(t)$ is the Gamma function.
\end{lemma}
 Now we derive Logarithmic uncertainty principle for the WOLCT.
\begin{theorem}
Let $\phi\in S(\mathbb{R})$. Then, for every $f\in S(\mathbb{R})$,
we have:
\begin{eqnarray}
	\int_{\mathbb{R}^{2}}\ln |u||V_{\phi}^{A}f(u,w)|^{2}\rm{d}\mit w\rm{d}\mit u+ \|\mit\phi\|^{2}
\int_{\mathbb{R}}\ln |t||\textit{f}(t)|^{2}\rm{d}\mit t
\geq(M+\ln |b|)\|\textit{f}\|^{2}\|\mit\phi\|^{2}
		\end{eqnarray}
where $M=\varphi(\frac{1}{2}-\ln \pi)$, $\varphi(t)=\frac{\rm{d}}{\rm{d}t}\ln [\Gamma(t)] $ and $\Gamma(t)$ is the Gamma function.
\end{theorem}
\begin{proof}
The proof of Theorem 8 is quite similar to the proof of Theorem 1 and so details are omitted.
  \end{proof}

\noindent  {\bf 4. Conclusions}

The uncertainty principles for many linear transforms such as the FT, the OLCT and the WLCT have been obtained. But the uncertainty principles for the WOLCT have not been studied.
In this paper, using the relationship between the WOLCT and the other transforms (the FT, the OLCT and the WLCT), the uncertainty principles for the WOLCT such as Heisenberg uncertainty principle, Hardy's uncertainty principle, Beurling's uncertainty principle, Lieb's uncertainty principle, Donoho-Stark's uncertainty principle, Amrein-Berthier-Benedicks's uncertainty principle, Nazarov's uncertainty principle and Logarithmic uncertainty principle are derived. These uncertainty principles are new results in WOLCT domain. In the further work, we will think about these applications of the WOLCT in signal processing and these uncertainty principles for discrete domain.

\medskip

\noindent  {\bf Acknowledge}

This work is supported by the National Natural Science Foundation of China (No. 61671063), and also by the Foundation for Innovative Research Groups of the National Natural Science Foundation of China (No. 61421001).


\end{document}